\documentclass[12pt]{article}
\usepackage[utf8]{inputenc}
\usepackage[T1]{fontenc}
\usepackage{array}
\usepackage{amsmath, amsthm}
\usepackage{amssymb}
\usepackage{enumitem}

\newtheorem{theorem}{Theorem}[section]
\newtheorem{proposition}[theorem]{Proposition}
\newtheorem{corollary}[theorem]{Corollary}

\newtheorem{lemma}[theorem]{Lemma}

\newtheorem{definition}[theorem]{Definition}

\newtheorem{remark}[theorem]{Remark}

\newcommand{\rwc}{R} 
\DeclareMathOperator{\rws}{RW} 
\newcommand{\Alpha}{A}
\DeclareMathOperator{\Factor}{Fac}
\DeclareMathOperator{\Pal}{Pal}
\newcommand{\ff}{f} 
\newcommand{\cc}{c} 
\newcommand{\TT}{T} 

\DeclareMathOperator{\LUF}{luf} 

\begin{document}

\title{Property of upper bounds on the number of rich words}

\author{Josef Rukavicka\thanks{Department of Mathematics,
Faculty of Nuclear Sciences and Physical Engineering, Czech Technical University in Prague
(josef.rukavicka@seznam.cz).}}

\date{\small{December 18, 2022}}

\maketitle

\begin{abstract}
A finite word $w$ is called \emph{rich} if it contains $\vert w\vert+1$ distinct palindromic factors including the empty word. 
Let $q\geq 2$ be the size of the alphabet. Let $R(n)$ be the number of rich words of length $n$. Let $d>1$ be a real constant and let $\phi, \psi$ be real functions such that
\begin{itemize}\item there is $n_0$ such that $2\psi(2^{-1}\phi(n))\geq d\psi(n)$ for all $n>n_0$, \item $\frac{n}{\phi(n)}$ is an upper bound on the palindromic length of rich words of length $n$, and
\item $\frac{x}{\psi(x)}+\frac{x\ln{(\phi(x))}}{\phi(x)}$ is a strictly increasing concave function.
\end{itemize}
We show that if $c_1,c_2$ are real constants and 
$R(n)\leq q^{c_1\frac{n}{\psi(n)}+c_2\frac{n\ln(\phi(n))}{\phi(n)}}$ then for every real constant $c_3>0$ there is a positive integer $n_0$ such that for all $n>n_0$ we have that
\[R(n)\leq q^{(c_1+c_3)\frac{n}{d\psi(n)}+c_2\frac{n\ln(\phi(n))}{\phi(n)}(1+\frac{1}{c_2\ln{q}}+c_3)}\mbox{.}\]

\end{abstract}

\section{Introduction}
A finite word is a sequence of letters of some given alphabet. A finite word $w$ is called a \emph{palindrome} if the word $w$ is equal to its reversal; for example ``radar'' or ``elle''. A finite word $w$ is called \emph{rich} if it contains $\vert w\vert+1$ distinct palindromic factors including the empty word. In the recent years, quite many articles can be found that deal with rich words; see, for instance, \cite{BuLuGlZa2, DrJuPi, GlJuWiZa, RUKAVICKA2021richext}.

One of the open questions being researched is the enumeration of rich words. 
Let $\rwc(n)$ denote the number of rich words of length $n$. In \cite{RukavickaRichWords2017}, it has been proved that the function $\rwc(n)$ grows subexponentially; formally \begin{equation}\label{ud789eehrk9}\lim_{n\rightarrow\infty}(\rwc(n))^{\frac{1}{n}}=1\mbox{.}\end{equation} 
It implies that for every real constant $c>1$ we have that 
\[\lim_{n\rightarrow\infty}\frac{\rwc(n)}{c^{n}}=0\mbox{.}\]

The formula (\ref{ud789eehrk9}) presents the currently best known upper bound on the number of rich words. In \cite{GuShSh15} it was conjectured for rich words on the binary alphabet that for some infinitely growing function $g(n)$ the following holds true  \[{\rwc(n)} = \mathcal{O} \Bigl(\frac{n}{g(n)}\Bigr)^{\sqrt{n}}\mbox{.}\]

The result (\ref{ud789eehrk9}) implies that there is subexponential upper bound on $\rwc(n)$. However currently we have no example a subexponential upper bound on the number of rich words $\rwc(n)$. In the current paper we present a property of an upper bound on $\rwc(n)$ that could give a hint how to derive a subexponential upper bound on $\rwc(n)$. To be able to formally state our main result, we introduce some more notation.

In \cite[Definition $4$ and Proposition $3$]{DrJuPi} it was proved that the longest palindromic suffix of a rich word $w$ has exactly one occurrence in $w$. Using this property it was proved that:

\begin{lemma} (see \cite[Lemma $1$]{RukavickaRichWords2017}) 
\label{lemma_A}
Let $w$ be a finite nonempty rich word. There exist distinct nonempty palindromes $w_1,w_2,\dots,w_p$ such that
\[\begin{split} 
w=w_pw_{p-1}\cdots w_1
 \mbox{ and }w_i \mbox{ is the longest palindromic suffix of }\\ w_pw_{p-1}\cdots w_i \mbox{ for }i=1,2,\dots ,p\mbox{.}
\end{split}\]
\end{lemma}
Based on the Lemma \ref{lemma_A}, the so-called \emph{UPS-factorization} and the function $\LUF(w)$ (the length of UPS-factorization) were introduced.
\begin{definition}(see \cite[Definition $2$]{RukavickaRichWords2017} and \cite[Definition $1.2$]{RUKAVICKA202295}) 
We define UPS-factorization (Unioccurrent Palindromic Suffix factorization) to be the factorization of a finite nonempty rich word $w$ into the form $w=w_pw_{p-1}\cdots w_1$ from Lemma \ref{lemma_A}.
Let $\LUF(w)=p$ be \emph{the length of UPS-factorization} of the rich word $w$.
\end{definition}

We define a set of strictly increasing concave functions.   
Let \[\begin{split}\Delta=\{\ff:\mathbb{R}^+\rightarrow\mathbb{R}^+\mid \ff'(x), \ff''(x)\mbox{ exist and }\ff'(x)>0 \\ \mbox{ and }
\ff''(x)<0
\mbox{ for all }x\geq 1\}\mbox{.}\end{split}\]

We define another set of functions. Let $\phi\in\Delta$. and let \[\begin{split}\TT_{\phi}=\{\psi:\mathbb{R}^+\rightarrow\mathbb{R}^+\mid \psi(x)\leq x\mbox{ and }\ff\in\Delta\mbox{, where }\\ \ff(x)=\frac{x}{\psi(x)}+\frac{x\ln(\phi(x))}{\phi(x)}\mbox{ for all }x\geq 1\}\mbox{.}\end{split}\]

Let $q\geq 2$ be the size of the alphabet, let $d>1$ be a real constant, and let $\phi\in\Delta, \psi\in\TT_{\phi}$ be such that
\begin{itemize}\item there is $n_0$ such that $2\psi(2^{-1}\phi(n))\geq d\psi(n)$ for all $n>n_0$ and \item $\frac{n}{\phi(n)}\geq \LUF(w)$ for every rich word with $\vert w\vert=n$; it means that $\frac{n}{\phi(n)}$ is an upper bound on the palindromic length of rich words of length $n$.
\end{itemize}

Our main result is the following theorem that allows to improve an upper bound on $\rwc(n)$.
\begin{theorem}
\label{djufirkf9dd}
We show that if $c_1,c_2$ are real constants and 
\[R(n)\leq q^{c_1\frac{n}{\psi(n)}+c_2\frac{n\ln(\phi(n))}{\phi(n)}}\] then for every real constant $c>0$ there is a positive integer $n_0$ such that for all $n>n_0$ we have that
\[R(n)\leq q^{(c_1+c_3)\frac{n}{d\psi(n)}+c_2\frac{n\ln(\phi(n))}{\phi(n)}(1+\frac{1}{c_2\ln{q}}+c_3)}\mbox{.}\]
\end{theorem}

\begin{remark}
The currently best known upper bound on the function $\LUF$ has been shown in \cite{RUKAVICKA202295}: There are real positive constants $\mu, \pi$ such that if $w$ is a finite nonempty rich word $n=\vert w\vert$ then \[\LUF(w)\leq \mu\frac{n}{e^{\pi\sqrt{\ln{n}}}}\mbox{.}\]
\end{remark}

\begin{remark}
Note that Theorem \ref{djufirkf9dd} improves the upper bound $q^{c_1\frac{n}{\psi(n)}+c_2\frac{n\ln(\phi(n))}{\phi(n)}}$ if
$c_1\frac{n}{\psi(n)}>c_2\frac{n\ln(\phi(n))}{\phi(n)}\mbox{}$ and
$\frac{c_1+c_3}{d}<c_1$.
\end{remark}

\section{Preliminaries}
Let $\mathbb{N}_0$ denote the set of all non-negative integers, let $\mathbb{N}^+$ denote the set of all positive integers, let $\mathbb{R}$ denote the set of all real numbers, and let $\mathbb{R}^+$ denote the set of all positive real numbers. 

Let $\Alpha$ denote a finite alphabet and let $q=\vert \Alpha\vert$. Let $\Alpha^n$ denote the set of all words of length $n\in\mathbb{N}^+$ over the alphabet $\Alpha$. Let $\Alpha^+=\bigcup_{i\geq 1}\Alpha^i$, let $\epsilon$ denote the empty word, and let $\Alpha^*=\{\epsilon\}\cup\Alpha^+$.  

Let $\Pal\subseteq\Alpha^*$ denote the set of all palindromes and let $\rws\subseteq \Alpha^*$ denote the set of all rich words. We have that $\epsilon\in\Pal\cap\rws$.

Let $\Factor_w\subseteq \Alpha^*$ be the set of all finite factors of the word $w\in\Alpha^*\cup\Alpha^{\infty}$. 
Let $\Factor_w(n)=\vert \Factor_w\cap\Alpha^n\vert$. The function $\Factor_w(n)$ is the \emph{factor complexity} of the word $w$.

\section{Previous results}

Let 
$\kappa(n)=\left\lceil\cc\frac{n}{\ln{n}}\right\rceil$. Using $\kappa(n)$, the following recurrence relation for the number of rich words is known. 
\begin{theorem}(see \cite[Theorem $7$]{RukavickaRichWords2017})
\label{label_th_D}
If $n\geq 2$, then
\begin{equation*}
R(n)\leq \sum_{p=1}^{\kappa(n)}\sum_{\substack{n_1+n_2+\dots +n_p=n \\ n_1,n_2,\dots, n_p\geq 1}}R\left(\left\lceil\frac{n_1}{2}\right\rceil\right)R\left(\left\lceil\frac{n_2}{2}\right\rceil\right)\cdots R\left(\left\lceil\frac{n_p}{2}\right\rceil\right)\mbox{.}
\end{equation*}
\end{theorem}
 
Let $\tau(n)$ be an upper bound on palindromic factorization of rich words of length $n$. From the proof of Theorem \ref{label_th_D} (see the proof of Theorem $7$ in \cite{RukavickaRichWords2017}) it is obvious the upper bound $\kappa(n)$ can be replaced by any function that forms an upper bound on $\LUF(w)$, where $w$ is a rich word. Thus from Theorem \ref{label_th_D} we have the following corollary.
\begin{corollary}
\label{duif90ej09e}
\begin{equation*}
R(n)\leq \sum_{p=1}^{\tau(n)}\sum_{\substack{n_1+n_2+\dots +n_p=n \\ n_1,n_2,\dots, n_p\geq 1}}R\left(\left\lceil\frac{n_1}{2}\right\rceil\right)R\left(\left\lceil\frac{n_2}{2}\right\rceil\right)\cdots R\left(\left\lceil\frac{n_p}{2}\right\rceil\right)\mbox{.}
\end{equation*}
\end{corollary}

From the proof of Proposition $9$ in \cite{RukavickaRichWords2017} we have that 
\begin{lemma}
\label{cbdjd87ejff}
If $L,n\in\mathbb{N}^+$ then 
\[\sum_{p=1}^{L}\sum_{\substack{n_1+n_2+\dots +n_p=n \\ n_1,n_2,\dots, n_p\geq 1}}1\leq \left(\frac{\mathrm{e}n}{L}\right)^{L}\mbox{.}\]
\end{lemma}

\section{Subadditive functions}

Given a strictly increasing concave function $\ff\in\Delta$ and a positive integer $n\geq 2$, the next technical lemma shows that the function $\ff(n-x)+\ff(x)$ has the global maximum in $x=\frac{n}{2}$.
\begin{lemma}
\label{cbdf998ejfo}
If $\ff\in\Delta$, $n\in\mathbb{N}^+$, and $n\geq 2$ then the function $\widehat\ff_n(x)=\ff(n-x)+\ff(x)$ has the global maximum in $x=\frac{n}{2}$; it means that $\widehat\ff_n(\frac{n}{2})> \widehat\ff_n(x)$ for every $x\in\mathbb{R}^+$ with $x\not=\frac{n}{2}$.
\end{lemma}
\begin{proof}
From the rule for derivation of an inner function, we have that 
\[\widehat \ff'_n(x)=\frac{\partial \widehat\ff_n(x)}{\partial x}=\frac{\partial \ff(n-x)}{\partial x}+\frac{\partial\ff(x)}{\partial x}=-\frac{\partial \ff(n-x)}{\partial (n-x)}+\frac{\partial\ff(x)}{\partial x}\mbox{.}\] It follows that $\widehat\ff'_n(\frac{n}{2})=0$. Since $\ff'(x)$ exists and $\ff'(x)>0$, it is easy to see that there is only one $x\in\mathbb{R}^+$ such that $\widehat \ff'_n(x)=0$ and $x>0$. From $\ff''(x)<0$ it follows that if $x\not=y$ then $\ff'(x)\not=\ff'(y)$; in consequence $\widehat\ff'_n(x)$ has both the global maximum in $x=\frac{n}{2}$. 

This completes the proof.
\end{proof}

We show how to apply Lemma \ref{cbdf998ejfo} for constructing an upper bound on the sum of values $\ff(x_i)$, where $\ff$ is a strictly increasing concave function.
\begin{corollary}
\label{d8989f8jedj}
If $\ff \in\Delta$, $k\in\mathbb{N}^+$, and $x_1,x_2,\dots,x_k\in\mathbb{R}^+$ then 
\[\sum_{i=1}^k\ff(x_i)\leq k\ff\left(\frac{x_1+x_2+\dots+x_k}{k}\right)\mbox{.}\] 
\end{corollary}
\begin{proof}
Let $x_1,x_2,\dots,x_k$ be such that if $y_1,y_2,\dots,y_k\in\mathbb{R}^+$ and $\sum_{i=1}^kx_i=\sum_{i=1}^ky_i$ then  
\begin{equation}\label{djf89lgfkjg}\sum_{i=1}^k\ff(x_i)\geq\sum_{i=1}^k\ff(y_i)\mbox{.}\end{equation}

Suppose that 
 there are $i,j\in\{1,2,\dots,k\}$ such that $x_i\not=x_j$. Then 
Lemma \ref{cbdf998ejfo} implies that \[\ff(x_i)+\ff(x_j)<2\ff\left(\frac{x_i+x_j}{2}\right)\mbox{.}\] This is a contradiction to (\ref{djf89lgfkjg}). We conclude that $x_i=x_j$ for all $i,j\in\{1,2,\dots,k\}$. This implies for every $i\in\{1,2,\dots,k\}$ that \[\sum_{i=1}^k\ff(x_i)=k\ff(x_i)=k\ff\left(\frac{x_1+x_2+\dots+x_k}{k}\right)\mbox{.}\]
The corollary follows. This completes the proof.
\end{proof}

We will need the following technical lemma. We omit the proof.
\begin{lemma}
\label{cbhd77hfie}
There is $x_0\in\mathbb{R}^+$ such that: If $x,y\in\mathbb{R}^+$, $\ln{x}\geq 1$, and $x_0< x\leq y$ then \[\frac{\ln{x}}{x}\geq \frac{\ln{y}}{y}\mbox{.}\]
\end{lemma}

\section{Upper bound for the number of rich words}

Let $\tau(x)=\frac{x}{\phi(x)}$. Fix constants $c_1,c_2\in\mathbb{R}^+$. Given $\psi\in\TT_{\phi}$, we define a function $\Omega:\mathbb{R}^+\rightarrow\mathbb{R}^+$ as follows
\[\Omega(x)=q^{c_1\frac{x}{\psi(x)}+c_2\tau(x)\ln{(\phi{(x)})}}\mbox{.}\]

We show an upper bound for the product of values of $\Omega(x)$.
\begin{proposition}
\label{cb88shd9jwe}
If  $n,p,n_1,n_2,\dots,n_p\in\mathbb{N}^+$, $n\geq p$, and $n=n_1+n_2+\dots+n_{p}$ then 
\[\Omega\left(\left\lceil\frac{n_1}{2}\right\rceil\right)\Omega\left(\left\lceil\frac{n_2}{2}\right\rceil\right)\cdots \Omega\left(\left\lceil\frac{n_{p}}{2}\right\rceil\right)\leq \left(\Omega\left(\frac{n}{2p}+1\right)\right)^p\mbox{.}\]
\end{proposition}
\begin{proof}
We have that \begin{equation}\begin{split}\label{bndh87ejfdy8}\left\lceil\frac{n_1}{2}\right\rceil+\left\lceil\frac{n_2}{2}\right\rceil+\dots+\left\lceil\frac{n_p}{2}\right\rceil\leq \\ \frac{n_1}{2}+\frac{n_2}{2}+\dots+\frac{n_p}{2}+p= \\ \frac{n}{2}+p\mbox{.}\end{split}\end{equation}

Let $\ff(x)=c_1\frac{x}{\psi(x)}+c_2\tau(x)\ln{(\phi{(x)})}$. 
From the definition of $\TT_{\phi}$ we have that $\ff\in\Delta$ and $\Omega(x)=q^{\ff(x)}$. Hence from Corollary \ref{d8989f8jedj} and (\ref{bndh87ejfdy8}) it follows that 
\[\begin{split}\Omega\left(\left\lceil\frac{n_1}{2}\right\rceil\right)\Omega\left(\left\lceil\frac{n_2}{2}\right\rceil\right)\cdots \Omega\left(\left\lceil\frac{n_p}{2}\right\rceil\right)=\\ q^{\ff\left(\left\lceil\frac{n_1}{2}\right\rceil\right)+\ff\left(\left\lceil\frac{n_2}{2}\right\rceil\right)+\dots +\ff\left(\left\lceil\frac{n_p}{2}\right\rceil\right)}\leq  \\ q^{p\ff\left(\frac{\frac{n}{2}+p}{p}\right)} \leq\\ \left(q^{\ff\left(\frac{n}{2p}+1\right)}\right)^p =\\ \left(\Omega\left(\frac{n}{2p}+1\right)\right)^p\mbox{.}\end{split}\]

This completes the proof.
\end{proof}

We show an another property of the function $\Omega$.
\begin{proposition}
\label{dy8fiifuf99}
If $p,n\in\mathbb{N}^+$ then \[\left(\Omega\left(\frac{n}{2p}+1\right)\right)^p\leq \left(\Omega\left(\frac{n}{2(p+1)}+1\right)\right)^{p+1}\mbox{.}\]
\end{proposition}
\begin{proof}
Let $\ff(x)=c_1\frac{x}{\psi(x)}+c_2\tau(x)\ln{(\phi{(x)})}$. It follows from Lemma \ref{cbdf998ejfo} that 
\begin{equation}\begin{split}\label{vbhd8djeh8}\left(\Omega\left(\frac{n}{2p}+1\right)\right)^p=\\  \left(q^{\ff\left(\frac{n}{2p}+1\right)}\right)^p=  \\ q^{p\ff\left(\frac{n}{2p}+1\right)}= \\ q^{\ff\left(\frac{n}{2p}+1\right)+(p-1)\ff\left(\frac{n}{2p}+1\right)}\leq \\ q^{2\ff\left(\frac{1}{2}(\frac{n}{2p}+1)\right)+(p-1)\ff\left(\frac{n}{2p}+1\right)}\mbox{.}\end{split}\end{equation}

Corollary \ref{d8989f8jedj} implies that 
\begin{equation}\begin{split}\label{cnbd7fd9jhe}2\ff\left(\frac{1}{2}\left(\frac{n}{2p}+1\right)\right)+(p-1)\ff\left(\frac{n}{2p}+1\right)\leq \\ (p+1)\ff_n\left(\frac{1}{p+1}\left(2\frac{1}{2}\left(\frac{n}{2p}+1\right)+(p-1)\left(\frac{n}{2p}+1\right)\right)\right)=\\ (p+1)\ff_n\left(\frac{1}{p+1}p\left(\frac{n}{2p}+1\right)\right)\leq \\ (p+1)\ff_n\left(\frac{n}{2(p+1)}+\frac{p}{p+1}\right)\leq   \\ (p+1)\ff_n\left(\frac{n}{2(p+1)}+1\right)\mbox{.}\end{split}\end{equation}

From (\ref{vbhd8djeh8}) and (\ref{cnbd7fd9jhe}) we have that 
\[\left(\Omega\left(\frac{n}{2p}+1\right)\right)^p\leq q^{ (p+1)\ff_n\left(\frac{n}{2(p+1)}+1\right)}=\left(\Omega\left(\frac{n}{2(p+1)}+1\right)\right)^{p+1}\mbox{.}\]
This completes the proof.

\end{proof}

We have that \begin{equation}\label{cbdjdu6eded}\frac{n}{2\tau(n)}=\frac{n}{2\left\lceil\frac{n}{\phi(n)}\right\rceil}\leq \frac{n}{2\frac{n}{\phi(n)}}=2^{-1}\phi(n)\mbox{.}\end{equation}


\begin{lemma}
\label{bvgdj984krfj9}
There is $n_0\in\mathbb{N}^+$ such that for all $n>n_0$ we have that
\[\tau(\phi(n))\ln{(\phi{(\phi(n))})} \leq \ln{(\phi(n))}\mbox{.}\]
\end{lemma}
\begin{proof}
Obviously $\phi(\phi(n))<\phi(n)$ for sufficiently large $n$. In addition from Lemma \ref{cbhd77hfie} it follows that for sufficiently large $n$ we have that \[\frac{\ln{(\phi(\phi(n)))}}{\phi(\phi(n))}\leq \frac{\ln{(\phi(n))}}{\phi(n)}\] and in consequence 
\begin{equation}\label{dhfuje998f}\frac{\phi(n)}{\phi(\phi(n))}\ln{(\phi(\phi(n)))}\leq \frac{\phi(n)}{\phi(n)}\ln{(\phi(n))}=\ln{(\phi(n))}\mbox{.}\end{equation}
From (\ref{dhfuje998f}) it follows that
\[\begin{split} \tau(\phi(n))\ln{(\phi{(\phi(n))})}&\leq \frac{\phi(n)}{\phi(\phi(n))}\ln{(\phi(n))}\\ &\leq \ln{(\phi(n))}\mbox{.}\end{split}\]
This completes the proof.
\end{proof}

\begin{proof}[Proof of Theorem \ref{djufirkf9dd}]
From Corollary \ref{duif90ej09e} and $R(n)\leq \Omega(n)$ we have that 
\begin{equation}
\label{hdufjflor334f}
\begin{split}
R(n)&\leq \sum_{p=1}^{\tau(n)}\sum_{\substack{n_1+n_2+\dots +n_p=n \\ n_1,n_2,\dots, n_p\geq 1}}R\left(\left\lceil\frac{n_1}{2}\right\rceil\right)R\left(\left\lceil\frac{n_2}{2}\right\rceil\right)\cdots R\left(\left\lceil\frac{n_p}{2}\right\rceil\right)\\ &\leq \sum_{p=1}^{\tau(n)}\sum_{\substack{n_1+n_2+\dots +n_p=n \\ n_1,n_2,\dots, n_p\geq 1}}\Omega\left(\left\lceil\frac{n_1}{2}\right\rceil\right)\Omega\left(\left\lceil\frac{n_2}{2}\right\rceil\right)\cdots \Omega\left(\left\lceil\frac{n_p}{2}\right\rceil\right)\mbox{.}\end{split}\end{equation}

From Proposition \ref{cb88shd9jwe} and (\ref{hdufjflor334f}) it follows that 
\begin{equation}
\label{duuje99jde}
\begin{split}
R(n)&\leq \sum_{p=1}^{\tau(n)}\sum_{\substack{n_1+n_2+\dots +n_p=n \\ n_1,n_2,\dots, n_p\geq 1}}\left(\Omega\left(\frac{n}{2p}+1\right)\right)^p\\&\leq \sum_{p=1}^{\tau(n)}\left(\left(\Omega\left(\frac{n}{2p}+1\right)\right)^p\sum_{\substack{n_1+n_2+\dots +n_p=n \\ n_1,n_2,\dots, n_p\geq 1}}1\right)\mbox{.}\end{split}\end{equation}

From Lemma \ref{cbdjd87ejff}, Proposition \ref{dy8fiifuf99}, and (\ref{duuje99jde}) it follows that
\begin{equation}
\label{bvd8whfbd3jw}
\begin{split}
R(n) &\leq \left(\Omega\left(\frac{n}{2\tau(n)}+1\right)\right)^{\tau(n)} \sum_{p=1}^{\tau(n)}\sum_{\substack{n_1+n_2+\dots +n_p=n \\ n_1,n_2,\dots, n_p\geq 1}}1 \\&\leq \left(\Omega\left(\frac{n}{2\tau(n)}+1\right)\right)^{\tau(n)} \left(\frac{\mathrm{e}n}{\tau(n)}\right)^{\tau(n)}\mbox{.}\end{split}\end{equation}

From (\ref{cbdjdu6eded}) and (\ref{bvd8whfbd3jw}) we have that 
\begin{equation}
\label{hdjdube83h9r}
\begin{split}
R(n) &= \left(\Omega\left(2^{-1}\phi(n)+1\right)\right)^{\tau(n)} (\mathrm{e}\phi(n))^{\tau(n)}\\ &= \left(\Omega\left(2^{-1}\phi(n)+1\right)\right)^{\tau(n)} q^{\frac{\tau(n)}{\ln{q}}(\ln{(\phi(n))}+1)}\mbox{.}\end{split}\end{equation}

Obviously there is $n_0$ such that for all $n>n_0$ we have that \begin{equation}\begin{split}\label{ug89rjjhfur}2^{-1}\phi(n)+1\leq \phi(n)\quad\mbox{ and }\quad \\ 2^{-1}\phi(n)\leq 2^{-1}\phi(n)+1\leq 2^{-1}(\phi(n)+1)\mbox{.}\end{split}\end{equation}

It follows from the definition of $\Omega$,  (\ref{hdjdube83h9r}), and (\ref{ug89rjjhfur}) that 
\begin{equation}
\label{ghds8fd9fje}
\begin{split}
R(n) &\leq \left(q^{c_1\frac{2^{-1}\phi(n)+1}{\psi(2^{-1}\phi(n)+1)}+c_2\tau(2^{-1}\phi(n)+1)\ln{(\phi{(2^{-1}\phi(n)+1)})}}\right)^{\tau(n)}q^{\frac{\tau(n)}{\ln{q}}(\ln{(\phi(n))}+1)}\\&\leq \left(q^{c_1\frac{2^{-1}\phi(n)+1}{\psi(2^{-1}\phi(n)+1)}+c_2\tau(\phi(n))\ln{(\phi{(\phi(n))})}}\right)^{\tau(n)}q^{\frac{\tau(n)}{\ln{q}}(\ln{(\phi(n))}+1)}\\ &\leq q^{c_1\tau(n)\frac{2^{-1}\phi(n)+1}{\psi(2^{-1}\phi(n)+1)}+c_2\tau(n)\tau(\phi(n))\ln{(\phi{(\phi(n))})}\tau(n)}q^{\frac{\tau(n)}{\ln{q}}(\ln{(\phi(n))}+1)}\mbox{.}\end{split}
\end{equation}

From Lemma \ref{bvgdj984krfj9} and (\ref{ghds8fd9fje}) we have that 
\begin{equation}
\label{bdd87fkreifj}
\begin{split}
R(n)&\leq q^{c_1\tau(n)\frac{2^{-1}(\phi(n)+1)}{\psi(2^{-1}\phi(n))}+c_2\tau(n)\ln{(\phi(n))}\tau(n)}q^{\frac{\tau(n)}{\ln{q}}(\ln{(\phi(n))}+1)} \\ &\leq q^{c_1\frac{\phi(n)+1}{\phi(n)}\frac{n}{2\psi(2^{-1}\phi(n))}+c_2\tau(n)\ln{(\phi(n))}}q^{\frac{\tau(n)}{\ln{q}}(\ln{(\phi(n))}+1)}\\ &\leq q^{c_1\frac{\phi(n)+1}{\phi(n)}\frac{n}{2\psi(2^{-1}\phi(n))}} q^{c_2\tau(n)\ln{(\phi(n))}(1+\frac{1}{c_2\ln{q}}+\frac{1}{c_2\tau(n)\ln{\phi(n)}\ln{q}})}\mbox{.}
\end{split}
\end{equation}

Obviously 
\begin{equation}
\label{bsv762dhosjd}
\begin{split}\lim_{n\rightarrow\infty}\frac{q^{c_1\frac{\phi(n)+1}{\phi(n)}\frac{n}{2\psi(2^{-1}\phi(n))}} q^{c_2\tau(n)\ln{(\phi(n))}(1+\frac{1}{c_2\ln{q}}+\frac{1}{c_2\tau(n)\ln{\phi(n)}\ln{q}})}}{q^{c_1\frac{n}{2\psi(2^{-1}\phi(n))}} q^{c_2\tau(n)\ln{(\phi(n))}(1+\frac{1}{c_2\ln{q}})}}=1\mbox{.}\end{split}\end{equation}
The theorem follows from (\ref{bdd87fkreifj}) and (\ref{bsv762dhosjd}). This completes the proof.
\end{proof}

\section*{Acknowledgments}
This work was supported by the Grant Agency of the Czech Technical University in Prague, grant No. SGS20/183/OHK4/3T/14.

\bibliographystyle{siam}
\IfFileExists{biblio.bib}{\bibliography{biblio}}{\bibliography{../!bibliography/biblio}}

\end{document}